\documentclass[10pt,a4paper,twoside,review]{article}

\usepackage{amssymb}
\usepackage{amsmath}
\usepackage{amsthm}
\usepackage[margin=2.5cm]{geometry}

\def\R{\mathbb{R}}
\def\C{\mathbb{C}}
\def\H{\mathbb{H}}

\def\la{\lambda}

\def\g{\mathfrak{g}}
\def\h{\mathfrak{h}}
\def\p{\mathfrak{p}}

\def\gl{\mathfrak{gl}}
\def\sl{\mathfrak{sl}}
\def\so{\mathfrak{so}}

\def\su{\mathfrak{su}}
\def\u{\mathfrak{u}}
\def\sp{\mathfrak{sp}}

\def\kalg{\mathfrak{k}}

\def\Dbdle{\mathcal{D}}

\def\Gbdle{\mathcal{G}}

\def\:{\lrcorner}
\def\#{\sharp}

\def\tens{\otimes}

\def\dsum{\oplus}

\def\prod{\times}

\def\isom{\cong}

\theoremstyle{plain}
\newtheorem{Theorem}{Theorem}[section]
\newtheorem{Lemma}[Theorem]{Lemma}

\newtheorem{Proposition}[Theorem]{Proposition}

\theoremstyle{definition}
\newtheorem{Definition}{Definition}[section]

\newtheorem*{TheoremA}{Theorem A}

\def\no{\noindent}

{\begin{list}{}%
         {\setlength{\leftmargin}{#1}}%
         \item[]%
}
{\end{list}}

\begin{document}

\title{Transitive conformal holonomy groups}

\author{Jesse Alt}

\maketitle

\begin{abstract}
For $(M,[g])$ a conformal manifold of signature $(p,q)$ and dimension at least three, the conformal holonomy group $\mathrm{Hol}(M,[g]) \subset O(p+1,q+1)$ is an invariant induced by the canonical Cartan geometry of $(M,[g])$. We give a description of all possible connected conformal holonomy groups which act transitively on the M\"obius sphere $S^{p,q}$, the homogeneous model space for conformal structures of signature $(p,q)$. The main part of this description is a list of all such groups which also act irreducibly on $\R^{p+1,q+1}$. For the rest, we show that they must be compact and act decomposably on $\R^{p+1,q+1}$, in particular, by known facts about conformal holonomy the conformal class $[g]$ must contain a metric which is locally isometric to a so-called special Einstein product.
\end{abstract}




\section{Introduction}

If $(M,[g])$ is a $C^{\infty}$ manifold endowed with a conformal class of semi-Riemannian metrics $[g]$ of signature $(p,q)$, then for $p + q \geq 3$ we have a well-defined invariant $\mathrm{Hol}(M,[g])$ called its \emph{conformal holonomy} (cf. Def. \ref{conf hol def}). Conformal holonomy groups have been intensively studied in recent years as basic invariants of the canonical conformal Cartan connection and thus of conformal structures. In contrast to the semi-Riemannian holonomy $\mathrm{Hol}(M,g) \subset O(p,q)$ for some choice of $g \in [g]$, the conformal holonomy is naturally identified as a subgroup of $O(p+1,q+1)$. This is a consequence of the fact that for conformal geometry, no canonical connection of the conformal structure $(M,[g])$ can be defined on (a reduction of) the linear frame bundle; rather, a canonical Cartan connection $\omega = \omega^{[g]}$ is defined on a reduction of the second order frame bundle and $\mathrm{Hol}(M,[g]) := \mathrm{Hol}(\omega^{[g]})$. Hence a number of new features and challenges appear in the study of conformal holonomy, both in terms of obtaining classification results and geometrically interpreting conformal holonomy reduction.\\

Initial results about conformal holonomy concerned the geometric meaning of conformal holonomy groups which preserve some subspace of $\R^{p+1,q+1}$ under the standard action. For example, a $\mathrm{Hol}(M,[g])$-invariant line $\R \cdot v \subset \R^{p+1,q+1}$ corresponds, for some open dense $M_0 \subset M$, to the existence of an Einstein metric in the conformal class $[g_{\vert M_0}]$, with the sign of the scalar curvature related to the causality of $v$. At least when the vector $v$ has non-zero length, this result follows from fundamental properties of parallel sections of the so-called standard tractor bundle in conformal geometry which were known for a long time, cf. \cite{BEG} and references therein. If $v$ is a null vector, the fact that $\mathrm{Hol}(M,[g])$-invariance of $\R \cdot v$ implies the existence of a parallel standard tractor was shown by T. Leistner in \cite{Leist}.\\

A generalization of this fact is the following: a $\mathrm{Hol}(M,[g])$-invariant decomposition $\R^{p+1,q+1} = V \dsum W$ via non-degenerate subspaces $V, W$ of respective dimensions $r+1,s+1 \geq 2$ corresponds, for some open dense subset $M_0 \subset M$, to a metric $g_0 \in [g_{\vert M_0}]$ which is locally isometric to a product of Einstein metrics of dimensions $r$ and $s$ with Einstein constants satisfying a certain relation (a ``special Einstein product'', cf. e.g. Theorem 1.2 of \cite{Arm} for the relation). This result was discovered independently by F. Leitner \cite{Felipe1} and S. Armstrong \cite{Arm}. It provides a rough analog of the de Rham/Wu Decomposition Theorem for pseudo-Riemannian holonomy.\\

The types of conformal holonomy described above are called \emph{decomposable}. The results on decomposable conformal holonomy have been used to derive classification results for conformal Riemannian holonomy, cf. \cite{Felipe1, Arm}, but it should be noted that those classifications do not account for the ``singular set'' (i.e. the complement of $M_0$) which might occur. Recently, a complete (global) classification of conformal holonomy in Riemannian signature, including classification of the possible singularities, has been given in \cite{ArmLeit}.\\

On the other hand, and in contrast to the corresponding problem for Riemannian holonomy groups, \emph{irreducible} conformal holonomy groups (i.e. those which act irreducibly on $\R^{p+1,q+1}$, leaving no non-trivial subspace invariant under the standard action) play no role in classifying conformal holonomy groups in Riemannian signature. This is a result of an algebraic fact: the only connected, irreducible subgroup $G \subset O(p+1,1)$ is $SO_0(p+1,1)$ (cf. \cite{DiOlm}, as well as \cite{DiLeistNeuk}). A classification of the connected, irreducible subgroups of $O(p+1,2)$ has also been obtained, giving a short list of possible connected, irreducible conformal holonomy groups in Lorentzian signature, cf. \cite{DiLeist}.\\

The aim of the present text is to give a classification of possible irreducible conformal holonomy groups for arbitrary signature, but under the additional assumption of \emph{transitivity}. A subgroup $H \subset O(p+1,q+1)$ has a natural action on the conformal M\"obius sphere $S^{p,q} \approx (S^p \times S^q)/\mathbb{Z}_2$, which we identify with the projectivized null-cone: $$S^{p,q} = \mathbb{P}(\mathcal{N}) \isom O(p+1,q+1)/\widetilde{P},$$ where $\mathcal{N} := \{ x \in \R^{p+1,q+1} : \vert x \vert = 0 \}$ and $\widetilde{P} \subset O(p+1,q+1)$ is the stabilizer of some real null line $\ell \subset \R^{p+1,q+1}$. We call a conformal holonomy group transitive if this action is transitive. The main result is:

\begin{TheoremA} Let $H = \mathrm{Hol}(M,[g]) \subset O(p+1,q+1)$ be a connected conformal holonomy group, for a conformal manifold of signature $(p,q)$ (with $p+q \geq 3$), and assume $H$ acts transitively on the conformal M\"obius sphere $S^{p,q}$. If $H$ acts irreducibly on $\R^{p+1,q+1}$, then it is isomorphic to one of the following:

(i) $SO_0(p+1,q+1)$ for all $p, q$;

(ii) $SU(n+1,m+1)$ for $p=2n+1, q=2m+1$;

(iii) $Sp(1)Sp(n+1,m+1)$ for $p=4n+3, q=4m+3$;

(iv) $Sp(n+1,m+)$ for $p=4n+3, q=4m+3$;

(v) $Spin_0(1,8)$ for $p=q=7$;

(vi) $Spin_0(3,4)$ for $p=q=3$;

(vii) $G_{2,2}$ for $p = 3, q = 2$.

\no If $H$ does not act irreducibly on $\R^{p+1,q+1}$, then it is compact and $(M,[g])$ has decomposable conformal holonomy. In particular, there exists $g_0 \in [g]$ which is locally isometric to a special Einstein product.
\end{TheoremA}

In Section 2, we recall the definition and summarize some relevant facts about conformal holonomy groups. Section 3 then gives the proof of Theorem A, the main step of which is to classify connected, \emph{semi-simple} subgroups $H \subset O(p+1,q+1)$ which act irreducibly on $\R^{p+1,q+1}$ and transitively on $S^{p,q}$. We conclude in Section 4 with a discussion of how this classification compares with the cases of irreducible conformal holonomy groups which have been studied in the literature. To our knowledge, the irreducible conformal holonomy groups studied to date are all transitive, and they are related to Fefferman-type constructions. For example, in \cite{CapGover1, CapGover2} the conformal holonomy group $SU(n+1,m+1) \subset O(2n+2,2m+2)$ (case (ii) in Theorem A) is shown to correspond to the original Fefferman construction due to \cite{Fef76}, where a natural conformal metric is defined on an $S^1$ bundle over a non-degenerate CR manifold of signature $(n,m)$. Except for cases (iii) and (v) in Theorem A, the conformal holonomy is known to be related to such a Fefferman construction, and the literature on these is reviewed in Section 4, where we also announce results of work in progress on one of the remaining cases, (v).

\section{Background on conformal holonomy}

For an arbitrary Lie group $G$ and a closed subgroup $P \subset G$, a Cartan geometry of type $(G,P)$ is given by the data $(\pi: \Gbdle \rightarrow M,\omega)$, where $M$ is a smooth manifold of dimension $\mathrm{dim}(G/P)$, $\pi: \Gbdle \rightarrow M$ is a $P$-principal bundle over $M$, and $\omega \in \Omega^1(\Gbdle,\g)$, called the Cartan connection, is a smooth $\g$-valued $1$-form (for $\g$ the Lie algebra of $G$) which satisfies an equivariance condition, respects the fundamental vector fields corresponding to the subgroup $P$, and gives a point-wise linear isomorphism between the tangent spaces of $\Gbdle$ and $\g$ (for a general reference on Cartan geometries and parabolic geometries, the reader is referred to \cite{CSbook}). The most common convention for defining the holonomy of a Cartan connection is the following: The inclusion $P \subset G$ determines an associated $G$-principal bundle $\widehat{\Gbdle} := \Gbdle \times_P G$, and this carries a unique $G$-principal connection $\widehat{\omega} \in \Omega^1(\widehat{\Gbdle};\g)$ determined by the condition $\iota^*\widehat{\omega} = \omega$, where $\iota: \Gbdle \hookrightarrow \widehat{\Gbdle}$ is the natural inclusion $\iota: u \mapsto [(u,e_G)]$. Then one defines, for $u \in \Gbdle$, the group $\mathrm{Hol}_u(\omega) := \mathrm{Hol}_{\iota(u)}(\widehat{\omega})$ as usual via the $\widehat{\omega}$-horizontal lifts of curves in $M$ to $\widehat{\Gbdle}$.\\

In particular, a conformal manifold $(M,[g])$ of signature $(p,q)$ and dimension $p+q \geq 3$ has a canonical Cartan geometry $(\pi: \Gbdle \rightarrow M,\omega^{[g]})$ of type $(\overline{G},\overline{P})$ for $\overline{G} = \mathbb{P}O(p+1,q+1) := O(p+1,q+1)/\{\pm Id\}$ and $\overline{P} \subset \overline{G}$ the parabolic subgroup which is the image under the quotient map of the stabilizer of a null line $\ell \subset \R^{p+1,q+1}$. If we write $\ell = \R v$ for some non-zero null vector $v \in \mathcal{N} \subset \R^{p+1,q+1}$, then the projection $G := O(p+1,q+1) \rightarrow \overline{G}$ restricts to an isomorphism $P \isom \overline{P}$, where $P \subset G$ is the subgroup of elements which preserve the null ray $\R_+v$. In this way, one identifies $\overline{P} \isom P \subset O(p+1,q+1)$ and defines $\widehat{\Gbdle} := \Gbdle \times_P G$ in order to consider the conformal holonomy of $(M,[g])$ to be a subgroup of $G = O(p+1,q+1)$ rather than of its quotient $\overline{G}$:

\begin{Definition} \label{conf hol def} For a choice of points $x \in M$ and $u \in \Gbdle_x$, the conformal holonomy of $(M,[g])$ with respect to $x$ and $u$ is given by $\mathrm{Hol}_x^u(M,[g]) := \mathrm{Hol}_u(\omega^{[g]}) \subset O(p+1,q+1).$ The abstract group defined up to isomorphism by the conjugacy class of $\mathrm{Hol}_x^u(M,[g])$ in $O(p+1,q+1)$ is denoted $\mathrm{Hol}(M,[g])$.
\end{Definition}

The M\"obius sphere $S^{p,q} = \mathbb{P}(\mathcal{N}) \isom \overline{G}/\overline{P}$ is identified with the set of null lines in $\R^{p+1,q+1}$. The double covering $S^p \times S^q \rightarrow S^{p,q}$ (which is non-trivial for $p,q>0$) is then realized by the projection $G/P \rightarrow \overline{G}/\overline{P}$, noting the diffeomorphism $G/P \approx S^p \times S^q$ given by identifying the set of null rays $\{ \ell_+ = \R_+ v \subset \R^{p+1,q+1} : v \in \mathcal{N} \}$ with $S^p \times S^q$. Of course, for $p,q \geq 2$, this is the universal cover of $S^{p,q}$.\\

We note one surprising fact about unitary conformal holonomy, which is interesting to contrast to semi-Riemannian holonomy theory:

\begin{Theorem} \label{unitary holonomy} (Leitner, \cite{Felipe2}) If $(M,[g])$ is a conformal manifold of signature $(2n+1,2m+1)$ and dimension at least $4$, and $\mathrm{Hol}(M,[g]) \subset U(n+1,m+1)$, then the connected component of $\mathrm{Hol}(M,[g])$ is contained in $SU(n+1,m+1)$.
\end{Theorem}

We get the following useful lemma as a corollary (cf. Lemma 62 of \cite{altdiss} for a proof):

\begin{Lemma} \label{semi-simple holonomy} If $H = \mathrm{Hol}(M,[g])$ is a connected conformal holonomy group of a conformal manifold of signature $(p,q)$ which acts irreducibly on $\R^{p+1,q+1}$, then $H$ is semi-simple.
\end{Lemma}

\section{Proof of Theorem A}

The main step in proving Theorem A is the following proposition:

\begin{Proposition} \label{irreducible H} Let $H \subset O(p+1,q+1)$ be a connected, semi-simple Lie subgroup which acts irreducibly on $\R^{p+1,q+1}$ and transitively on $S^{p,q}$. Then $H$ is isomorphic to one of the groups (i)-(vii) in Theorem A. Conversely, all of those subgroups act irreducibly on $\R^{p+1,q+1}$ and transitively on $S^{p,q}$.
\end{Proposition}

Before proving Proposition \ref{irreducible H}, we prove a result covering the case where $H$ does not act irreducibly:

\begin{Proposition} \label{reducible H} If $H \subset O(p+1,q+1)$ is a closed subgroup which acts transitively on $S^{p,q}$ but does not act irreducibly on $\R^{p+1,q+1}$, then $H$ must be contained in a maximal compact subgroup $K \isom SO(p+1) \times SO(q+1)$. In particular, $H$ is compact and $\R^{p+1,q+1} \isom \R^{p+1} \dsum \R^{q+1}$ is decomposable as an $H$-module.
\end{Proposition}

\begin{proof} Let $V \subset \R^{p+1,q+1}$ be a non-trivial $H$-invariant subspace. Then $V \cap \mathcal{N} = \{ 0 \}$. Otherwise, we would have $\mathcal{N} \subset V$ by transitivity of $H$ on $S^{p,q} = \mathbb{P}(\mathcal{N})$, since $V$ is a subspace. But $\mathcal{N} \backslash \{ 0 \}$ is a \emph{full} submanifold of $\R^{p+1,q+1}$, being the orbit of a point under the action by $O(p+1,q+1)$, which acts irreducibly on $\R^{p+1,q+1}$ (cf. e.g. Prop. 4 of \cite{DiLeist}). This means $\mathcal{N}$ is not contained in any proper affine subspace of $\R^{p+1,q+1}$. Therefore, since $V$ contains no non-zero null vectors, the restriction of the metric to $V$ is definite, in particular non-degenerate, so we get a $H$-invariant decomposition $\R^{p+1,q+1} = V \dsum V^{\perp}$. Similarly, the restriction of the metric to $V^{\perp}$ must also be definite, so we must have $V \dsum V^{\perp} \isom \R^{p+1} \dsum \R^{q+1}$, which shows that $H \subset K \isom SO(p+1) \times SO(q+1)$.\end{proof}

\no \emph{Proof of Proposition \ref{irreducible H}.} We consider three cases: $p \geq 3, q = 0$ (corresponding to Riemannian signature); $p \geq 2, q = 1$ (corresponding to Lorentzian signature); and $p,q \geq 2$. In the first two cases, Proposition \ref{irreducible H} is already a corollary of Theorem 1.1 of \cite{DiOlm} and Theorem 1 of \cite{DiLeist}, respectively. The first of these results, already mentioned in the Introduction, says that the only connected subgroup of $O(p+1,1)$ acting irreducibly on $\R^{p+1,1}$ is $SO_0(p+1,1)$. The second result says that the only connected subgroups of $O(p+1,2)$ acting irreducibly on $\R^{p+1,2}$ are isomorphic to: $SO_0(p+1,2)$ for general $p$; $SU(n+1,1)$, $U(n+1,1)$ or $S^1 \cdot SO_0(n+1,1)$ for $p=2n+1$ odd; and $SO_0(2,1)_i \subset O(3,2)$ for $p=2$. Of the semi-simple groups from this classification, $SO_0(2,1)_i \subset O(3,2)$ does not act transitively on $S^{2,1}$, because there are null lines in $S^{2,1}$ where it does not act locally transitively, cf. Appendix A.1 of \cite{DiLeist}.\\

So from now on, we can restrict consideration to connected, semi-simple subgroups $H \subset G := SO_0(p+1,q+1)$ which act irreducibly on $\R^{p+1,q+1}$ and transitively on $S^{p,q}$, with $p,q \geq 2$. Our strategy for proving that $H$ must be among the list claimed in Theorem A under these assumptions is as follows: First we show that a maximal compact subgroup $K$ of $H$ (or some cover of $K$) must act transitively and effectively on the product of spheres $S^p \times S^q$. A result of B. Kamerich gives a list of all possibilities for $K$ or its covering group (cf. Theorem \ref{compact transitive list} below). Since $H$ is semi-simple, we can use this list and the standard tables giving maximal compact subgroups of simple Lie groups to enumerate the possible groups $H$ which could occur. Then, working at the Lie algebra level, we use standard methods from representation theory to exclude all but a few of these possibilities by the criteria that $\h$ must have an irreducible real representation of dimension $p+q+2$ which preserves a metric of signature $(p+1,q+1)$ (cf. Lemma \ref{reducing the possibilities}). For the remaining cases, those which do not occur in the list of Theorem A are excluded by a simple additional argument.\\

First, note that under the current assumptions on $H$, its universal cover $\tilde{H}$ must act transitively on $S^p \times S^q$ (which is the universal cover of $S^{p,q}$, as noted in Section 2) by a standard result on transitive group actions (cf. e.g. Proposition 6 in Chapter 1 of \cite{Oni}). Hence, all maximal compact subgroups $\tilde{K} \subset \tilde{H}$ act transitively on $S^p \times S^q$ as a result of the following well-known proposition:

\begin{Proposition} (Montgomery, \cite{Mont}) If $H$ is a connected Lie group which acts transitively on a compact, simply connected manifold $\mathbf{X}$, then all maximal compact subgroups $K \subset H$ also act transitively on $\mathbf{X}$.
\end{Proposition}

Clearly, a maximal compact subgroup $\tilde{K} \subset \tilde{H}$ is a covering space of some maximal compact subgroup $K \subset H$. Recall that a group action $a: H \times \mathbf{X} \rightarrow \mathbf{X}$ is called \emph{effective} if the subgroup of elements of $H$ which act by the identity on $\mathbf{X}$ consists of only the identity element $e \in H$. Since $K \subset H$ is compact and $H$ is closed in $G$, $K$ is a compact subgroup of $G$ and hence contained in one of its maximal compact subgroups, which are all isomorphic to $SO(p+1) \times SO(q+1)$. The latter is known to act effectively (and transitively) on $S^p \times S^q$, and thus it follows that the elements of $\tilde{K}$ which act on $S^p \times S^q$ by the identity form at most a (discrete) subgroup of the Galois group of the covering $\tilde{K} \rightarrow K$. In particular, for every compact subgroup $K \subset H$, some covering space of $K$ must act transitively and effectively on $S^p \times S^q$. This allows us to apply the following classification result on compact groups acting transitively on $S^p \times S^q$, due to B. Kamerich (cf. \cite{Oni} for a discussion and proof; the list given here is derived from Ch. 5, Theorem 6 of that reference using the centralizers to give all groups acting transitively and effectively, cf. Prop. 3.6 and relevant tables in \cite{Kramer}):

\begin{Theorem} \label{compact transitive list} (Kamerich, \cite{Kam}) If $K$ is a connected, compact Lie group acting transitively and effectively on $S^p \times S^q$ for $p,q \geq 2$, then $K$ and $S^p \times S^q$ are isomorphic to one of the following:\\
\no List (I):

(a) $K = SU(4)$ acting on $S^5 \times S^7$;

(b) $K = Spin(8)$ acting on $S^7 \times S^7$;

(c) $K = Spin(7)$ or $U(1)Spin(7)$ acting on $S^6 \times S^7$;

(d) $K = Spin(8)$ acting on $S^6 \times S^7$;

(e) $K = SO(8)$ or $U(1)SO(8)$ acting on $S^6 \times S^7$.

\no List (II), where $K = K_1 \times K_2$ and the $K_i$ acting transitively on $S^{p_i}$ for $p_1 := p$, $p_2 := q$, are:

(a) $K_i = SO(p_i+1)$ acting on $S^{p_i}$;

(b) $K_i = SU(\frac{p_i + 1}{2})$ or $U(\frac{p_i + 1}{2})$ acting on $S^{p_i}$;

(c) $K_i = Sp(\frac{p_i+1}{4})$ or $Sp(1)Sp(\frac{p_i+1}{4})$ acting on $S^{p_i}$;

(d) $K_i = Spin(9)$ acting on $S^{15}$;

(e) $K_i = Spin(7)$ acting on $S^7$;

(f) $K_i = G_2$ acting on $S^6$.

\no List (III):

(a) $K = SU(2)SU(2), SU(2)U(2)$ or $U(2)U(2)$ acting on $S^3 \times S^2$;

(b) $K = SU(\frac{p+1}{2})SU(2), SU(\frac{p+1}{2})U(2)$ or $U(\frac{p+1}{2})U(2)$ acting on $S^p \times S^2$;

(c) $K = Sp(\frac{p+1}{4})SU(2), Sp(\frac{p+1}{4})U(2)$ or $Sp(1)Sp(\frac{p+1}{4})U(2)$ acting on $S^p \times S^2$;

(d) $K = Sp(\frac{p+1}{4})SU(3)$ or $Sp(\frac{p+1}{4})U(3)$ acting on $S^p \times S^5$;

(e) $K = Sp(\frac{p+1}{4})Sp(2)$ or $K = Sp(1)Sp(\frac{p+1}{4})Sp(2)$ acting on $S^p \times S^7$.
\end{Theorem}

We can now use this to limit the possible Lie algebras $\h$ of our irreducible transitive semi-simple subgroup $H \subset SO_0(p+1,q+1)$:

\begin{Lemma} \label{reducing the possibilities} Let $\h$ be a semi-simple Lie algebra with maximal compact subalgebra $\kalg \subset \h$ isomorphic to the Lie algebra of one of the compact Lie groups $K$ in the lists (I)-(III) of Theorem \ref{compact transitive list}. If $\rho: \h \rightarrow \gl(V)$ is a faithful real irreducible representation of dimension $p+q+2$ for the integers $p,q$ corresponding to $K$, and such that $\rho(\h) \subset \so(V,g)$ for some non-degenerate, symmetric bilinear form $g$ of signature $(p+1,q+1)$, then $\h$, $\kalg$ and $\rho$ are, up to isomorphism, one of the following, with the given restrictions on $p, q$:

$\mathbf{(I_{b,i})}$ $\h = \so(1,8)$, $\rho = \tau_0(\lambda_4)$, $\kalg = \so(8)$ and $p=q=7$;

$\mathbf{(I_{b,ii})}$ $\h = \so(8,\C)_{\R}$, $\rho \in \{ \tau_{\R}(1 \tens \lambda_1), \tau_{\R}(1 \tens \lambda_3), \tau_{\R}(1 \tens \lambda_4) \}$, $\kalg = \so(8)$ and $p=q=7$;

$\mathbf{(II_{a,a})}$ $\h = \so(p+1,q+1)$, $\rho = \tau_0(\la_1)$, $\kalg = \so(p+1) \dsum \so(q+1)$ and $p,q \geq 2$;

$\mathbf{(II_{a,b,i})}$ $\h = G_{2,2}$, $\rho = \tau_0(\la_1)$, $\kalg = \so(3) \dsum \su(2)$ and $p=3,q=2$;

$\mathbf{(II_{a,b,ii})}$ $\h = \so(3,4)$, $\rho = \tau_0(\la_3)$, $\kalg = \so(4) \dsum \su(2)$ and $p=q=3$;

$\mathbf{(II_{b,b,i})}$ $\h = \su(n+1,m+1)$, $\rho = \tau_{\R}(\la_1)$, $\kalg = \su(n+1) \dsum \u(m+1)$ and $p=2n+1,q=2m+1 \geq 3$;

$\mathbf{(II_{b,b,ii})}$ $\h = \su(2) \dsum \sl(2,\C)_{\R}$, $\rho = \tau_{\R}(\la_1 \tens (1 \tens \la_1))$, $\kalg = \su(2) \dsum \su(2)$ and $p=q=3$;

$\mathbf{(II_{b,b,iii})}$ $\h = \sl(2,\C)_{\R} \dsum \sl(2,\C)_{\R}$, $\rho = \tau_{\R}((1 \tens \la_1) \tens (1 \tens \la_1))$, $\kalg = \su(2) \dsum \su(2)$ and $p=q=3$;

$\mathbf{(II_{c,c,i})}$ $\h = \sp(n+1,m+1)$, $\rho = \tau_{\R}(\la_1)$, $\kalg = \sp(n+1) \dsum \sp(m+1)$ and $p=4n+3,q=4m+3 \geq 3$;

$\mathbf{(II_{c,c,ii})}$ $\h = \sp(1) \dsum \sp(n+1,m+1)$, $\rho = \tau_0(\la_1 \tens \la_1)$, $\kalg = \sp(1) \dsum \sp(n+1) \dsum \sp(m+1)$ and $p=4n+3,q=4m+3 \geq 3$.
\end{Lemma}

\begin{proof} The idea of the proof is simple enough: For each compact group $K$ in the lists of Theorem \ref{compact transitive list}, we can consult the standard tables on real simple Lie algebras (cf. e.g. Appendix C, Section 3, of \cite{Knapp}) to determine the semi-simple Lie algebras $\h$ which have maximal compact sub-algebra isomorphic to $\kalg$; for each such $\h$, we use the techniques from representation theory of semi-simple Lie algebras to determine whether it admits a faithful irreducible representation $\rho$ of the appropriate dimension $d=p+q+2$, and whether $\rho$ preserves a non-degenerate symmetric bilinear form of signature $(p+1,q+1)$. When finished checking these criteria for all possibilities, we are left with precisely the above list.\\

In practice, this is an extremely tedious calculation, involving carrying out simple verifications with weights, but for over one hundred different cases. For this reason, we only outline the steps here and record the details for all the cases separately in \cite{transhol-details}.\\

First, note that $\h$ can never be compact. This follows from the well-known fact that every linear representation $R: K \rightarrow Gl(V)$ of a compact Lie group $K$ admits an invariant positive-definite symmetric bilinear form. But using a variation on the argument in the proof of Schur's Lemma, it is easy to prove that if a finite-dimensional irreducible module $V$ admits an invariant positive-definite metric, then it has no invariant metric of indefinite signature. (In fact, a strengthening is possible, cf. Theorem 3 of \cite{DiLeistNeuk}: If $H \subset Gl(V)$ acts irreducibly, then the space of $H$-invariant symmetric bilinear forms which are not of neutral signature is at most one-dimensional.)\\

Next, for each of the non-compact semi-simple $\h$ with maximal compact sub-algebra $\kalg$ on the list, note that in general the real irreducible representations $\rho: \h \rightarrow \gl(V)$ of dimension $d$ are divided into the following types: those for which the complexification $\rho(\C): \h \rightarrow \gl(V(\C))$ is irreducible (Type I); and those which are the underlying real representation of an irreducible complex representation $\tau: \h \rightarrow \gl_{\C}(W)$, i.e. $(\rho,V) = (\tau_{\R},W_{\R})$ (Type II). Thus we first need to determine, for each $\h$, the sets of faithful complex irreducible representations of (complex) dimensions $d$ and $d/2$ -- denoted $C_d(\h)$ and $C_{d/2}(\h)$, respectively. The complex irreps of $\h$ are determined up to isomorphism by a highest weight $\Lambda$ for the complexification $\h(\C)$, and information about various invariants including dimension can be computed from these weights and are collected in tables in the literature. In particular, the relevant non-empty sets $C_d(\h)$ and $C_{d/2}(\h)$ are obtained with the help of information gathered in Table 5 of \cite{OniVin} and 4.10-4.26 of \cite{Kramer}.\\

For each $\tau = \tau(\Lambda) \in C_d(\h)$ we must verify that $\tau$ corresponds to a real irrep of Type I, i.e. $\tau = \rho(\C)$ which we denote by $\rho = \tau_0 = \tau_0(\Lambda)$. This amounts by standard results to checking that $\tau$ is self-conjugate and has a real structure, and these properties can in turn be determined from the highest weight $\Lambda$. (For details, cf. e.g. Theorem 1 of \cite{Iwahori}, Sec. 3 of Reference Chapter in \cite{OniVin}, 2.3.14-2.3.15 of \cite{CSbook}.) In addition, for these $\tau$ we must have $\tau(\h(\C)) \subset \so(d,\C) = \so(p+1,q+1) \tens \C$ as a necessary criteria for the condition $\rho(\h) \subset \so(p+1,q+1)$. This corresponds to the criteria that $\tau$ is self-dual and that the non-degenerate bilinear form it preserves is symmetric, which can also be determined from $\Lambda$, cf. Exercises 4.3.5-4.3.13 of \cite{OniVin}.\\

Finally, on the other hand for each $\tau = \tau(\Lambda) \in C_{d/2}(\h)$ we must verify that $\tau$ corresponds to a real irrep of Type II, i.e. $\rho = \tau_{\R}$. This is the condition that either $\tau$ is not self-conjugate or, if it is self-conjugate, that it admits no real structure (i.e. the index must be $-1$), and these criteria are also checked via the highest weight $\Lambda$.\\

After testing these criteria for all possibilities in \cite{transhol-details}, we are left with only those possibilities listed in the statement of the lemma. The fundamental weights $\Lambda$ are indicated in terms of a basis of fundamental weights $\{ \la_1,\ldots,\la_k \}$ of the simple factors of the complexification $\h(\C)$ (where $1$ indicates a trivial representation), with a tensor symbol used to indicate taking a tensor product of representations corresponding to the highest weights of the simple factors of $\h(\C)$. Note that it is also indicated whether $\rho$ is of Type I or Type II, namely whether $\rho = \tau_0(\Lambda)$ or $\rho = \tau_{\R}(\Lambda)$, respectively. \end{proof}

From Lemma \ref{reducing the possibilities} the only cases which have to be dealt with in order to establish Theorem A are $\mathbf{(I_{b,ii})}$, $\mathbf{(II_{b,b,ii})}$ and $\mathbf{(II_{b,b,iii})}$. These $\rho$ are all of Type II, that is $\rho$ is the underlying real representation of some complex irreducible representation $\tau$, and for the $\tau$ in each case we can check using Exercises 4.3.5-4.3.13 of \cite{OniVin} that $\tau$ is self-dual and orthogonal, in particular we have $\tau(\h) = \rho(\h) \subset \so(d/2,\C)$ in all these cases. Note also that $p=q$ in these cases. Hence, we can apply the following result to exclude these remaining cases:

\begin{Lemma} \label{excluding so(n,C)} Let $H$ be a connected Lie group and $R: H \rightarrow Gl(V)$ a real representation of even dimension $2n$ with infinitesimal $\rho: \h \rightarrow \so(n,n) \subset \gl(V)$. If $\rho(\h) \subset \so(n,\C) \subset \so(n,n)$, then the induced action of $H$ on the M\"obius sphere $S^{n-1,n-1}$ is not transitive.
\end{Lemma}

\begin{proof} Let $V = \C^n \isom \R^{2n}$ and let $R, H$, etc. be as in the hypotheses. We let $\{ e_1,\ldots,e_n \}$ be an ordered basis of $V$ over $\C$ and denote by $<,>$ the standard non-degenerate, symmetric ($\C$-)bilinear form with respect to this basis. Let $B$ denote some real symmetric bilinear form of neutral signature which is compatible with $<,>$, i.e. such that $SO(n,\C) \isom SO_{\C}(V,<,>)$ is contained in $SO(V,B) \isom SO(n,n)$. By Proposition 5 of \cite{DiLeistNeuk}, $B$ must be given as a linear combination of the real and imaginary parts of $<,>$, so we have $B = \alpha\mathrm{Re}(<,>) +\beta\mathrm{Im}(<,>)$ for some $\alpha,\beta \in \R$.\\

We show that the induced action of $H$ on $S^{n-1,n-1}$ is not transitive by exhibiting a real $B$-null line $\ell \subset V$ for which the $H$-orbit is not open in $S^{n-1,n-1}$, i.e. such that $\rho(\h) + \p(\ell) \subsetneq \so(V,B)$, where $\p(\ell) := \mathfrak{stab}(\ell) \subset \so(V,B)$. We define $\ell = \R w$ for $w := ze_1 + \overline{z}e_n$, $z := (1 + i)/\sqrt{2}$. (Since $<w,w> = 0$, $w$ must be $B$-null.) Define a new basis $\{ f_1,\ldots,f_n \}$ for $V$ over $\C$ as follows: Let $f_1 := w$, $f_j := e_j$ for $2 \leq j \leq n-1$, and let $$f_n := (\alpha \overline{z} + \beta z)e_1 + (\alpha z - \beta \overline{z})e_n.$$ We then calculate the identities: $<f_1,f_1> = <f_n,f_n> = 0$, $<f_1,f_n> = \alpha + i\beta \neq 0$, $B(f_1,f_n) = \alpha^2 + \beta^2 \neq 0$, $B(f_1,if_n) = 0$; for $2 \leq j \leq n-1$, we have $<f_1,f_j> = <f_n,f_j> = 0$ and hence $B(f_1,f_j) = B(f_1,if_j) = B(f_n,f_j) = B(f_n,if_j) = 0$.\\

Thus, with respect to the basis $\{ f_1,\ldots,f_n \}$ the quadratic form of $<,>$ has entries $\alpha + i\beta$ in the bottom-left and top-right corners, and all other entries of the first and last rows and columns are zero. In particular, for any $X \in \h$, the matrix $\rho(X) \in \so_{\C}(V,<,>)$ must have a zero in the bottom-left corner with respect to this (complex) basis. Thus, with respect to the real basis $\{ f_1,\ldots,f_n, if_1,\ldots,if_n \}$, $\rho(X)$ must have a first column with zero's in the $n$th and $2n$th components. But, similarly, we see by looking at the quadratic form for $B$ with respect to this real basis that matrices of $\so(V,B)$ can have non-zero entries in all components of their first column except the $n$th one. In particular, since elements of $\p = \p(\R w)$ have a non-zero entry in only the first component of their first column, we have $\so(V,B) \supsetneq \rho(\h) + \p$, i.e. $R(H)$ does not act locally transitively at the point $\ell \in S^{n-1,n-1}$. \end{proof}

That concludes the proof of Proposition \ref{irreducible H}. In light of Lemma \ref{semi-simple holonomy}, most of Theorem A follows from Propositions \ref{irreducible H} and \ref{irreducible H}. The only claim which still needs to be verified is the existence of the metric $g_0 \in [g]$ which is a local special Einstein product (i.e. that $g_0$ is defined on all of $M$ rather than simply on an open dense subset as for general decomposable conformal holonomy). This follows since in the general case, the metric $g_0$ is defined from so-called normal conformal Killing forms given by the volume forms of the non-degenerate $\mathrm{Hol}(M,[g])$-invariant subspaces $V,W$ in the decomposition $\R^{p+1,q+1} = V \dsum W$. The singular set where $g_0$ fails to be defined is given by the set where the norm of these normal conformal Killing forms vanish. In our case, since the subspaces $V$ and $W$ are definite, this set equals the zero set of the normal conformal Killing forms. But by the discussion in Sec. 2.6 of \cite{CGH}, this zero set must be empty if the conformal Killing form is non-trivial, since $H$ acts transitively on $S^{p,q}$ and $H$ stabilizes the alternating form inducing the conformal Killing form.

\section{Conclusion and outlook}

Theorem A gives a partial restriction on irreducible conformal holonomy groups which complements the other results obtained up to now (for Riemannian and Lorentzian signature). Transitivity of $H \subset O(p+1,q+1)$ on $S^{p,q}$ gives one condition under which one can hope to carry out a Fefferman-type construction inducing a conformal structure of signature $(p,q)$ from a parabolic geometry of some type $(H,Q)$ (cf. Ch. 4.5 of \cite{CSbook} for the general set-up). One consequence of Theorem A is to show that such a Fefferman-type construction can be related to all irreducible connected conformal holonomy groups $H \subset O(p+1,q+1)$ which act transitively on $S^{p,q}$. In fact, except for cases (iii) and (v), all the groups given by Theorem A have been studied and results in the literature show that (locally, at least) conformal manifolds with the given holonomy correspond to such generalized conformal Fefferman spaces:\\

\no Case (ii): For $H = SU(n+1,m+1)$ and $Q = \mathrm{Stab}_H(\C v)$ for a complex null line $\C v \subset \C^{n+1,m+1}$, the conformal Fefferman space induced from a parabolic geometry of type $(H,Q)$ was studied in \cite{CapGover1, CapGover2} and gives an induced conformal structure on a natural $S^1$-bundle of a non-degenerate CR manifold.\\

\no Case (iv): For $H = Sp(n+1,m+1)$ and $Q = \mathrm{Stab}_H(\H v)$ for a quaternionic null line $\H v \subset \H^{n+1,m+1}$, the quaternionic analogue of (ii), the conformal Fefferman space was studied in \cite{altdiss, qcfefferman}; it gives an induced conformal structure on a natural $S^3$- or $SO(3)$-bundle of a quaternionic contact manifold.\\

\no Case (vi): For $H = Spin_0(3,4)$ and $Q = P \cap H$ the stabilizer of a real null line in $\R^{4,4}$, the Fefferman construction associates to a $6$-manifold $M$, endowed with a generic distribution $\Dbdle \subset TM$ of rank $3$, a conformal structure of signature $(3,3)$ on $M$. This was studied in \cite{Bryant, HamSag2}.\\

\no Case (vii): For $H = G_{2,2}$ (the non-compact real form of the simple group $G_2$) and $Q = P \cap H$ the stabilizer of a real null line in $\R^{4,3}$, the Fefferman construction associates to a $5$-manifold $M$, endowed with a generic distribution $\Dbdle \subset TM$ of rank $2$, a conformal structure of signature $(3,2)$ on $M$. Its study goes back to work of E. Cartan in the early 20th century, and more recently in \cite{Nurowski, HamSag} among others.\\

The new case (v), $H = Spin_0(1,8)$, is the subject of ongoing work by the author in collaboration with F. Leitner. In this case, the only non-trivial parabolic subgroup $Q$ of $H$ is the one given by the stabilizer of a null line in $\R^{1,8}$ under the representation $\lambda: Spin_0(1,8) \rightarrow SO_0(1,8)$. We have the general conditions needed for a Fefferman-type construction, which associates to a parabolic geometry of type $(H,Q)$ a conformal structure of signature $(7,7)$ on the total space of a natural bundle over the base space of the geometry of type $(H,Q)$. In other words, from a conformal Riemannain spin manifold $(M,[g],\sigma)$ of dimension $7$, the construction induces in a natural way a conformal metric of signature $(7,7)$ on the total space of a fiber bundle over $M$ (the fiber type is $S^7$). In contrast to the Fefferman-type constructions in the other cases discussed above, for this type $(H,Q)$ the induced Cartan connection of conformal type $(G,P)$ will never be normal unless the Cartan geometry of type $(H,Q)$ has vanishing curvature, in other words, unless the Riemannian spin $7$-manifold is conformally flat.\\

It should be emphasized that an irreducibly acting conformal holonomy group $\mathrm{Hol}(M,[g]) \subset O(p+1,q+1)$ need not, \emph{a priori}, act transitively on $S^{p,q}$. Indeed, the classification of \cite{DiLeist} gives the possible irreducible conformal holonomy group $SO_0(2,1)_i \subset O(3,2)$ for Lorentzian $3$-manifolds. Through work in progress in collaboration with A. J. Di Scala and T. Leistner, we can exclude this case, however, applying the results of \cite{CGH}.\\

In general, transitivity appears to be a rather restrictive condition to place on conformal holonomy, and much more work is evidently needed to obtain a conformal analogue of Berger's list. But studying a weakening of this condition, such as \emph{locally transitive} conformal holonomy (i.e. satisfying $\mathfrak{so}(p+1,q+1) = \mathfrak{hol}(M,[g]) + \mathfrak{p}$), might be a fruitful next step on the way toward that aim. Thanks to the results of \cite{CGH}, we now have important tools for studying the geometry in these cases.

\no \begin{small}SCHOOL OF MATHEMATICS, UNIVERSITY OF THE WITWATERSRAND, P O WITS 2050, JOHANNESBURG, SOUTH AFRICA.\end{small}\\
\no E-mail: \verb"jesse.alt@wits.ac.za"

\end{document}